\documentclass[a4paper]{amsart}

\usepackage{amsmath, amsthm, amssymb}
\usepackage[english]{babel}
\usepackage[utf8]{inputenc}
\usepackage{mathrsfs}                     
\usepackage{bbm}                          
\usepackage[all]{xy}                      
\usepackage{graphicx}                     
\usepackage{color}
\usepackage{url}
\usepackage{import}                       
\usepackage{todonotes}                    
\usepackage{algpseudocode}                
\usepackage[style=numeric, minnames=3,
            doi=false, url=false, isbn=false,
            firstinits=true,
            sortcites=true,
            backend=bibtex]{biblatex}      
\usepackage{enumerate}                    
\usepackage{framed}                       
\usepackage{csquotes}

\theoremstyle{plain}
\newtheorem{theorem}[subsection]{Theorem}
\newtheorem{proposition}[subsection]{Proposition}
\newtheorem{lemma}[subsection]{Lemma}
\newtheorem{corollary}[subsection]{Corollary}

\newtheorem*{theorem*}{Theorem}
\newtheorem*{proposition*}{Proposition}
\newtheorem*{lemma*}{Lemma}
\newtheorem*{corollary*}{Corollary}

\theoremstyle{definition}

\newtheorem{remark}[subsection]{Remark}

\newtheorem*{definition*}{Definition}
\newtheorem*{remark*}{Remark}
\newtheorem*{example*}{Example}
\newtheorem*{openquestion*}{Open Question}

\def\al{\alpha}

\def\de{\delta}

\def\et{\eta}
\def\th{\theta}
\def\io{\iota}

\def\ph{\varphi}
\def\ch{\chi}
\def\ps{\psi}
\def\om{\omega}

\def\De{\Delta}

\def\Om{\Omega}

    \let\ii=\i           
\def\inv{^{-1}}
\def\x{\times}
\def\p{\partial}

\def\R{{\mathbb R}}

\let\on=\operatorname

\let\mc=\mathcal
\let\mf=\mathfrak

\newcommand{\ud}{\,\mathrm{d}}

\def\todomartins#1{}

\newcommand{\executeiffilenewer}[3]{%
\ifnum\pdfstrcmp{\pdffilemoddate{#1}}%
{\pdffilemoddate{#2}}>0%
{\immediate\write18{#3}}\fi%
}

\title{Regularity of Maps between Sobolev Spaces}
\author{Martins Bruveris}
\date{\today}

\address{
Martins Bruveris, Department of Mathematics, 
Brunel University London, Ux\-bridge, UB8 3PH, United Kingdom}
\email{martins.bruveris@brunel.ac.uk}


\thanks{The author was partially supported by the Erwin Schr\"odinger Institute programme
Infinite-Dimensional Riemannian Geometry with Applications to Image Matching
and Shape Analysis.}

\begin{document}

\begin{abstract}
Let $F : H^q \to H^q$ be a $C^k$-map between Sobolev spaces, either on $\R^d$ or on a compact manifold. We show that equivariance of $F$ under the diffeomorphism group allows to trade regularity of $F$ as a nonlinear map for regularity in the image space: for $0 \leq l \leq k$, the map $F: H^{q+l} \to H^{q+l}$ is well-defined and of class $C^{k-l}$. This result is used to study the regularity of the geodesic boundary value problem for Sobolev metrics on the diffeomorphism group and the space of curves.
\end{abstract}

\maketitle

\section{Introduction}

The main result of this paper is inspired from and a generalisation of results on the smoothness of geodesics for right-invariant Riemannian metrics on the diffeomorphism group. Riemannian metrics on the diffeomorphism group have been studied since it was recognised in \cite{Arnold1966} that solutions of Euler's equations for incompressible fluids correspond to geodesics on the group of volume-preserving diffeomorphisms with respect to a right-invariant Riemannian metric. The well-posedness of Euler's equation was established in \cite{Ebin1970} by showing that the corresponding geodesic spray is a smooth vector field on the group $\mc D^q_{\on{vol}}(M)$ of volume-preserving $H^q$-diffeomorphisms for $q > \dim M/2 + 1$ and $M$ a compact manifold. A smooth geodesic spray on a Hilbert manifold gives rise to a smooth exponential map and because the metric is right-invariant, this exponential maps is $\mc D^q_{\on{vol}}(M)$-equivariant. 

The right-invariance of the exponential map was used in \cite{Ebin1970} to show the following result: if the initial conditions of a geodesic are of class $H^{q+k}$, then so is the whole geodesic. This property implies that smooth ($C^\infty$) initial conditions for Euler's equations have smooth solutions. The same property was observed for various other right-invariant Riemannian metrics on the diffeomorphism group as well as for reparametrisation invariant Riemannian metrics on the space of curves.

In this paper we want to prove a general version of this result, both for Sobolev spaces on Euclidean space and for Sobolev spaces on manifolds.

\begin{theorem*}
Let $q > \frac d2 + 1$, $0 \leq l \leq k$ and $F : H^q(\R^d, \R^n) \to H^q(\R^d, \R^m)$ be a $\mc D^q(\R^d)$-equivariant $C^k$-map, i.e. $F(u \circ \ph) = F(u) \circ \ph$. Then $F$ maps $H^{q+l}$ into $H^{q+l}$ and $F: H^{q+l}(\R^d, \R^n) \to H^{q+l}(\R^d, \R^m)$ is a $C^{k-l}$-map.
\end{theorem*}

Here and in the following we assume $q \in \R$ and $l,k \in \mathbb N$. The group $\mc  D^q(\R^d)$ is the group of $H^q$-diffeomorphisms; see Sect.~\ref{sec:Rd}.

Previously the strongest statement was that if $F: H^q \to H^q$ is a $\mc D^q$-equivariant $C^1$-map, then $F:H^{q+1} \to H^{q+1}$ is well-defined. No statement was made about the continuity or differentiability of the resulting map. Next we state the corresponding result for Sobolev spaces on manifolds. Let $M$ be a compact manifold and $N,P$ smooth manifolds, all without boundary.

\begin{theorem*}
Let $q > \dim M/ 2 + 1$, $0 \leq l \leq k$ and $F : H^q(M, N) \to H^q(M, P)$ be a $\mc D^q(M)$-equivariant $C^k$-map, i.e. $F(u \circ \ph) = F(u) \circ \ph$. Then $F$ maps $H^{q+l}$ into $H^{q+l}$ and $F: H^{q+l}(M, N) \to H^{q+l}(M, P)$ is a $C^{k-l}$-map.
\end{theorem*}
In a nutshell this results states that given an equivariant map $F:H^q\to H^q$ we can trade smoothness of the map to gain spatial smoothness of the image $F(u)$. If $F$ is a $C^\infty$-map, then it also induces a $C^\infty$-map between the spaces $H^\infty$.

\begin{corollary*}
Let $q > \frac d2+1$ and $F : H^q(\R^d, \R^n) \to H^q(\R^d, \R^m)$ be a $\mc D^q(\R^d)$-equivariant $C^\infty$-map. Then $F : H^\infty(\R^d, \R^n) \to H^\infty(\R^d, \R^m)$ is a $C^\infty$-map.
\end{corollary*}

The same can be also formulated for maps defined on compact manifolds.

\begin{corollary*}
Let $q > \dim M/2+1$ and $F : H^q(M, N) \to H^q(M, P)$ be a $\mc D^q(M)$-equivariant $C^\infty$-map. Then $F : C^\infty(M,N) \to C^\infty(M,P)$ is a $C^\infty$-map.
\end{corollary*}

In Sect.~\ref{sec:diff} we apply this theorem to study the regularity of the geodesic boundary value problem for right-invariant Riemannian metrics on the diffeomorphism group. We show that if $\ph_0$ and $\ph_1$ are nonconjugate along the geodesic $\ph(t)$, then the whole geodesic is as smooth as $\ph_0$ and $\ph_1$. In Sect.~\ref{sec:curves} we show the same result for reparametrization invariant Sobolev metrics on the space of curves.

\subsection*{Note}
We will write $u \lesssim v$, if the inequality $u \leq C v$ holds for some constant $C > 0$, that may depend on the parameters $q, k, l, d, m, n$ and the manifolds $M, N, P$ involved, but is independent of the functions $F, u, v$. The constant may also depend on the auxiliary functions $\om_i, h_j$ and $X_j$ introduced in the proofs and additional dependencies will be stated in the text.

\section{Differentiability in Banach Spaces}

For Banach spaces $E, F$ we denote by $\mc L^k_{\on{sym}}(E,F)$ the space of bounded, symmetric, $k$-linear mappings $E \x \dots \x E \to F$. Let $U \subseteq E$ be open. A function $f : U \to F$ is $C^1$, if it is Fr\'echet differentiable and the derivative $Df : U \to \mc L(E, F)$ is continuous.

The following lemma is standard and is stated without proof.

\begin{lemma}
\label{lem:reverse_taylor}
Let $E, F$ be Banach spaces and $U \subseteq E$ a convex, open set. Let $\al \in C(U, \mc L(E,F))$. Assume that $f : U \to F$ is a mapping, such that
\[
f(y) = f(x)+ \int_0^1 \al(x+t(y-x)).(y-x) \ud t\,,
\]
holds for all $x,y \in U$. Then $f \in C^{1}(U, F)$ and $D f = \al$.
\end{lemma}

The next lemma shows, that if a function is differentiable on a dense subspace of a Banach space and the derivatives can be extended to continuous maps on the bigger space, then the function is differentiable on the bigger space.

\begin{lemma}
\label{lem:extend_diff}
Let $E, F, G$ be Banach spaces, $E \subseteq F$ a dense subspace, $f \in C^k(U, G)$ with $U = V \cap E$ and $V \subseteq F$ open. If we can extend $f$ and its derivatives $D^j f$ to $D^j f \in C(V, \mc L^j_{\on{sym}}(F, G))$ for $0 \leq j \leq k$, then $f \in C^k(V, G)$.
\end{lemma}

\begin{proof}
We have for $x, y \in U$, and $0 \leq j \leq k-1$,
\begin{equation}
\label{eq:Dkf}
D^{j} f(y) = D^{j} f(x) + \int_0^1 D^{j+1} f(x+ t(y-x)).(y-x) \ud t\,.
\end{equation}
Since $D^j f$ can be extended to $D^j f \in C(V, \mc L^j_{\on{sym}}(F, G))$ and both sides in~\eqref{eq:Dkf} are continuous on $V$, the identity continues to hold for $x,y \in V$. 

Now we argue inductively: Since $f \in C(V, G)$ and $Df$ can be extended to $V$, we obtain by Lem.~\ref{lem:reverse_taylor} that $f \in C^1(V, G)$. Now we apply Lem.~\ref{lem:reverse_taylor} with $Df$ in place of $f$ to conclude that $Df \in C^1(V, \mc L(F, G))$ and so $f \in C^2(V, G)$. In this way we obtain inductively that $f \in C^k(V, G)$.
\end{proof}

\section{$\mc D^q(\R^d)$-equivariant Maps}
\label{sec:Rd}

The Sobolev spaces $H^q(\R^d)$ with $q \in \R$ can be defined in terms of the Fourier transform
\[ 
\mc F f(\xi) = (2\pi)^{-d/2} \int_{\R^d} e^{-i \langle  x,\xi\rangle} f(x) \ud x\,,
\]
and consist of $L^2$-integrable functions $f$ with the property that $(1+|\xi|^2)^{q/2} \mc F f$ is $L^2$-integrable as well. An inner product on $H^q(\R^d)$ is given by
\[
\langle f, g \rangle_{H^q} = \mf{Re} \int_{\R^d} (1 + |\xi|^2)^q \mc F f(\xi) \overline{\mc F g(\xi)} \ud \xi\,.
\]

Denote by $\on{Diff}^1(\R^d)$ the group of $C^1$-diffeomorphisms of $\R^d$, i.e.,
\[
\on{Diff}^1(\R^d) = \{ \ph \in C^1(\R^d,\R^d) \,:\, \ph \text{ bijective, } \ph\inv \in C^1(\R^d,\R^d) \}\,.
\] 
For $q > d/2+1$ and $q \in \R$ there are three equivalent ways of defining the group $\mc D^q(\R^d)$ of Sobolev diffeomorphisms:
\begin{align*}
\mc D^q(\R^d) &= \{ \ph \in \on{Id} + H^q(\R^d,\R^d) \,:\, \ph \text{ bijective, }
\ph\inv \in \on{Id} + H^q(\R^d,\R^d) \} \\
&= \{ \ph \in \on{Id} + H^q(\R^d,\R^d) \,:\, 
\ph \in \on{Diff}^1(\R^d) \} \\
&= \{ \ph \in \on{Id} + H^q(\R^d,\R^d) \,:\, 
\det D\ph(x) > 0,\, \forall x \in \R^d \}\,.
\end{align*}
If we denote the three sets on the right by $A_1$, $A_2$ and $A_3$, then it is not difficult to see the inclusions $A_1 \subseteq A_2 \subseteq A_3$. The equivalence $A_1 = A_2$ has first been shown in \cite[Sect. 3]{Ebin1970b} for the diffeomorphism group of a compact manifold; a proof for $\mc D^q(\R^d)$ can be found in \cite{Inci2013}. Regarding the inclusion $A_3 \subseteq A_2$, it is shown in \cite[Cor. 4.3]{Palais1959} that if $\ph \in C^1$ with $\det D\ph(x) > 0$ and $\lim_{|x |\to \infty} | \ph(x)| = \infty$, then $\ph$ is a $C^1$-diffeomorphism.

It follows from the Sobolev embedding theorem, that $\mc D^q(\R^d) - \on{Id}$ is an open subset of $H^q(\R^d,\R^d)$ and thus a Hilbert manifold. Since each $\ph \in \mc D^q(\R^d)$ has to decay to the identity as $|x|\to \infty$, it follows that $\ph$ is orientation preserving. More importantly, $\mc D^q(\R^n)$ is a topological group, but not a Lie group, since left-multiplication and inversion are continuous, but not smooth operations.

We will use the following two results regarding the multiplication and composition of functions in Sobolev spaces.

\begin{lemma}[{\cite[Lem.~2.3]{Inci2013}}]
Let $q, r \in \R$ with $q > \frac d2$ and $0 \leq r \leq q$. Then pointwise multiplication
\[
H^q(\R^d, \R) \times H^r(\R^d,\R) \to H^r(\R^d,\R)\,,\quad
(f, g) \mapsto f \cdot g\,,
\]
is a bounded bilinear map.
\end{lemma}

\begin{lemma}[{\cite[Thm.~1.1 and Rem.~1.5]{Inci2013}}]
\label{thm:diff_rd_comp}
Let $q > \frac d2 + 1$ and $k \in \mathbb N$. Then
\[
H^{q+k}(\R^d, \R^n) \times \mc D^q(\R^d) \to H^q(\R^d,\R^n)\,,\quad
(f, \ph) \mapsto f \circ \ph\,,
\]
is a $C^k$-map.
\end{lemma}

Now we are ready to state and prove the main theorem for $\R^d$.

\begin{theorem}
\label{thm:main_euclidean}
Let $q > \frac d2 + 1$, $0 \leq l \leq k$ and $F : H^q(\R^d, \R^n) \to H^q(\R^d, \R^m)$ be a $\mc D^q(\R^d)$-equivariant $C^k$-map, i.e. $F(u \circ \ph) = F(u) \circ \ph$. Then $F$ maps $H^{q+l}$ into $H^{q+l}$ and $F: H^{q+l}(\R^d, \R^n) \to H^{q+l}(\R^d, \R^m)$ is a $C^{k-l}$-map.
\end{theorem}

\begin{proof}
The proof is split into several steps.

{\bfseries Step 1.} If $F:H^q \to H^q$ is $C^1$, then $F: H^{q+1} \to H^{q+1}$ is $C^0$. \\
Let $(\Om_i)_{i \in \mathbb N}$ be a uniformly localy finite cover of $\R^d$ by open balls, $(\om_i)_{i \in \mathbb N}$ a subordinate smooth partition of unity with uniformly (in $i$) bounded derivatives, and $h_1, \dots, h_d$ a basis of $\R^d$. Then an equivalent norm for $H^q(\R^d,\R^n)$ is given by
\begin{equation}
\label{eq:norm_Rd}
\|u \|_{H^q} \sim \| u \|_{H^{q-1}} + \sum_{j=1}^d \sum_{i \in \mathbb N} \| \om_i Du.h_j\|_{H^{q-1}}\,;
\end{equation}
see \cite[Sect.~7.2.2]{Triebel1992}.

Let $\ps_t = \on{exp}(t. \om_i. h_j)$ be the one-parameter subgroup generated by $\om_i.h_j$, i.e., $\ps_t$ satisfies the ODE $\p_t \ps_t = \om_i.h_j \circ \ps_t$. The existence of $\ps_t$ -- nontrivial since $\mc D^q(\R^d)$ is not a Lie group -- is shown for example in \cite[Thm.~4.4]{Bruveris2014_preprint}. Then by Lem.~\ref{thm:diff_rd_comp} the map
\[
\R \x H^{q+1} \to H^q\,,\quad(t, u) \mapsto u \circ \ps_t
\]
is $C^1$. Now fix $u \in H^{q+1}$ and consider the identity $F(u \circ \ps_t) = F(u) \circ \ps_t$. Both sides of the identity are $C^1$ as maps $\R \to H^{q-1}$ and by differentiating at $t=0$ we obtain the identity
\[
DF(u).Du.\om_i h_j = D\left(F(u)\right). \om_i h_j\,,
\]
We can estimate the $H^q$-norm of the right hand side via the left hand side,
\[
\| \om_i D\left(F(u)\right).h_j\|_{H^q} \leq \|DF(u)\|_{\mc L(H^q, H^q)} \| \om_i Du.h_j\|_{H^q}\,,
\]
and hence we see using the equivalent $H^q$-norm \eqref{eq:norm_Rd} that
\begin{equation}
\label{eq:Fu_Hq1}
\| F(u) \|_{H^{q+1}} \lesssim \| F(u) \|_{H^q} + \| DF(u) \|_{\mc L(H^q, H^q)}
\| u \|_{H^{q+1}} \,.
\end{equation}
This shows that $F(u) \in H^{q+1}$, provided $u \in H^{q+1}$. Regarding continuity we can show in the same manner the estimate
\begin{multline}
\label{eq:Fuv_Hq1}
\| F(u) -F(v) \|_{H^{q+1}} \lesssim  \| F(u) - F(v) \|_{H^q} + \\ +
\| DF(u) -DF(v) \|_{\mc L(H^q, H^q)} \| u \|_{H^{q+1}} +  \| DF(v) \|_{\mc L(H^q, H^q)} \| u -v\|_{H^{q+1}}
\end{multline}
for $u, v \in H^{q+1}$, from which the continuity of $F: H^{q+1} \to H^{q+1}$ follows.

{\bfseries Step 2.} If $F : H^q \to H^q$ is a $C^k$-map, then $F : H^{q+k-1} \to H^{q+k-1}$ is $C^1$ and $F : H^{q+k} \to H^{q+k}$ is $C^0$. \\
We will show this together with the explicit estimates
\begin{align}
\label{eq:Fuv_Hqk}
  \| F(u) - F(v) \|_{H^{q+k}} &\lesssim \sum_{j=0}^k \| D^j F(u) - D^j F(v) \|_{\mc L^j(H^{q}, H^{q})}  +{} \\
\notag
&\qquad\qquad {}+\| D^j F(v) \|_{\mc L^j(H^q,H^q)} \| u - v \|_{H^{q+k}}\\
\label{eq:DFv_Hqkm1}
\| DF(v) \|_{\mc L(H^{q+k-1}, H^{q+k-1})} &\lesssim
\sum_{j=0}^k \| D^j F(v)\|_{\mc L^j(H^q,H^q)}\,,
\end{align}
which are valid for $u, v$ in a bounded $H^{q+k}$-ball, using induction on $k$. 

For $k=1$ this is step 1. Now assume the statement has been shown for $k$ and let $F : H^q \to H^q$ be $C^{k+1}$. Then $DF \in C^k(H^q, \mc L(H^q, H^q))$ and thus also $DF \in C^k(H^q \x H^q, H^q)$. Since $H^q(\R^d, \R^n) \x H^q(\R^d, \R^n) \cong H^q(\R^d, \R^{2n})$, we obtain by induction that $DF \in C^1(H^{q+k-1}, H^{q+k-1})$ and $DF \in C(H^{q+k} \x H^{q+k}, H^{q+k})$. We now take $v, w \in H^{q+k}$ lying in a bounded $H^{q+k}$-ball and apply~\eqref{eq:Fu_Hq1} to obtain
\begin{align*}
\| DF(v).w \|_{H^{q+k}} 
&\lesssim \| DF(v).w \|_{H^{q+k-1}} + \| D^2F(v).w \|_{\mc L(H^{q+k-1}, H^{q+k-1})} \\
&\lesssim \| DF(v) \|_{\mc L(H^{q+k-1}, H^{q+k-1})} + 
\sum_{j=0}^{k} \| D^{j+1} F(v).w\|_{\mc L^j(H^q,H^q)} \\
&\lesssim \sum_{j=0}^{k+1} \| D^j F(v)\|_{\mc L^j(H^q,H^q)}\,.
\end{align*}
From this we obtain
\[
\| DF(v) \|_{\mc L(H^{q+k},H^{q+k})}
\lesssim \sum_{j=0}^{k+1} \| D^j F(v)\|_{\mc L^j(H^q,H^q)}\,,
\]
which completes the induction for~\eqref{eq:DFv_Hqkm1}.

 The induction assumption~\eqref{eq:Fuv_Hqk} applied to the $C^k$-map $DF$ shows that for $u, v, w \in H^{q+k}$, lying in a bounded $H^{q+k}$-ball,
\begin{align*}
\| DF(u).w &- DF(v).w \|_{H^{q+k}} \lesssim \\ 
& \lesssim \sum_{j=0}^k \| D^{j+1} F(u).w - D^{j+1} F(v).w \|_{\mc L^j(H^q, H^q)} \\
&\qquad\qquad {}+ \| D^{j+1} F(u).w \|_{\mc L^{j}(H^q,H^q)} \| u - v \|_{H^{q+k}} \\
& \lesssim \sum_{j=1}^{k+1} \| D^{j} F(u) - D^{j} F(v) \|_{\mc L^{j}(H^q, H^q)}
+ \| D^{j} F(v)\|_{\mc L^j(H^q,H^q)} \| u-v\|_{H^{q+k}} \,;
\end{align*}
here we have used the module property of Sobolev spaces. Therefore $DF \in C(H^{q+k}, \mc L(H^{q+k}, H^{q+k}))$ with
\begin{multline}
\label{eq:DFuv_Hqk}
\| DF(u) - DF(v) \|_{\mc L(H^{q+k}, H^{q+k})} \lesssim \\
\lesssim \sum_{j=1}^{k+1} \| D^{j} F(u) - D^{j} F(v) \|_{\mc L^{j}(H^q, H^q)}
+ \| D^{j} F(v)\|_{\mc L^j(H^q,H^q)} \| u-v\|_{H^{q+k}}
\end{multline}
Since also $F \in C^k(H^q, H^q)$, by induction $F \in C(H^{q+k}, H^{q+k})$. Now we can apply Lem.~\ref{lem:extend_diff} to obtain that $F \in C^1(H^{q+k}, H^{q+k})$ and hence by step 1, we have $F \in C(H^{q+k+1}, H^{q+k+1})$ together with the estimate,
\begin{align*}
\| F(u) -F(v) \|_{H^{q+k+1}} &\lesssim
\| F(u) -F(v) \|_{H^{q+k}} + \| DF(u) -DF(v) \|_{\mc L(H^{q+k}, H^{q+k})} \\
&\qquad\qquad+ \| DF(v) \|_{\mc L(H^{q+k}, H^{q+k})} \| u - v \|_{H^{q+k+1}}\,.
\end{align*}
Now we combine the induction assumption~\eqref{eq:Fuv_Hqk} and the estimates~\eqref{eq:DFuv_Hqk} and~\eqref{eq:DFv_Hqkm1} to obtain
\begin{align*}
\| F(u) -F(v) \|_{H^{q+k+1}} &\lesssim
\sum_{j=0}^{k+1} \| D^j F(u) - D^j F(v) \|_{\mc L^j(H^{q}, H^{q})}  +{} \\
\notag
&\qquad\qquad {}+\| D^j F(v) \|_{\mc L^j(H^q,H^q)} \| u - v \|_{H^{q+k+1}}\,.
\end{align*}
This concludes the induction.

{\bfseries Step 3.} If $F : H^q \to H^q$ is $C^k$, then $F : H^{q+l} \to H^{q+l}$ is $C^{k-l}$. \\
The case $l=0$ is trivial and the case $l=k$ was proven in step 2. Now let $1 \leq l \leq k$. We consider $D^j F$ as a map $D^j F: H^q(\R^d, \R^{(j+1)n}) \to H^q(\R^d, \R^m)$. Then the maps
\[
F, DF, D^2 F, \dots, D^{k-l} F: H^q \to H^q \text{ are at least $C^l$}\,.
\]
Thus by step 2 we have
\[
F, DF, D^2 F, \dots, D^{k-l} F : H^{q+l} \to H^{q+l} \text{ are $C^0$}\,,
\]
and we have the additional inequality for $0 \leq j \leq k-l$ and $u, v$ in a bounded $H^{q+l}$-ball,
\begin{align*}
  \| D^j F(u) - D^j F(v) \|_{\mc L^j(H^{q+l},H^{q+l})} &\lesssim \sum_{j=0}^l \| D^{l+j} F(u) - D^{l+j} F(v) \|_{\mc L^{l+j}(H^{q}, H^{q})}  +{} \\
&\qquad\qquad {}+\| D^{l+j} F(v) \|_{\mc L^{l+j}(H^q,H^q)} \| u - v \|_{H^{q+l}}\,,
\end{align*}
which shows $D^j F \in C(H^{q+l}, \mc L^j(H^{q+l}, H^{q+l}))$. Thus by Lem.~\ref{lem:extend_diff} we obtain $F \in C^{k-l}(H^{q+l}, H^{q+l})$. This concludes the proof.
\end{proof}

\begin{remark}
All the arguments in the proof of the theorem are local in nature, i.e. the one-parameter subgroups $\on{exp}(t.\om_i.h_j)$ are only considered around $t=0$, the estimates for $\| F(u) - F(v) \|_{H^{q+k}}$ and $\| DF(v) \|_{H^{q+k-1}}$ are only required to hold for $u,v$ in bounded balls and the statement about differentiability itself is local. Hence the theorem can also be proven for functions defined on open subsets of Sobolev spaces.
\end{remark}

\begin{corollary}
Let $q > \frac d2 + 1$, $0 \leq l \leq k$, $V \subseteq H^q(\R^d,\R^n)$ an open subset and $F : V \to H^q(\R^d, \R^m)$ a $\mc D^q(\R^d)$-equivariant $C^k$-function, i.e. $F(u \circ \ph) = F(u) \circ \ph$ for $q \in \mc D^q(\R^d)$, whenever $u, u \circ \ph \in V$. Then $F$ maps $U = V \cap H^{q+l}$ into $H^{q+l}$ and $F: U \to H^{q+l}(\R^d, \R^m)$ is a $C^{k-l}$-map.
\end{corollary}

\section{$\mc D^q(M)$-equivariant Maps}

In this section we assume that
$M$ is a $d$-dimensional compact manifold and $N, P$ are $n$- and $m$-di\-men\-sio\-nal manifolds respectively, both without boundary.

To define the spaces $H^q(M,N)$ we require $q > \frac d2$. A continuous map $f:M\to N$ belongs to $H^q(M,N)$, if around each point $x \in M$, there exists a chart $\ch : \mc U \to U\subseteq \R^d$ of $M$ and a chart $\et : \mc V \to V\subseteq \R^n$ of $N$ around $f(x)$, such that $\et\circ f \circ\ch\inv \in H^q(U,\R^n)$. When $N = \R^n$, then $H^q(M,\R^n)$ is the Sobolev space of functions on a manifold and the condition $q > \frac d2$ is not necessary; see \cite{Aubin1998}. In general $H^q(M,N)$ is not a vector space, but $H^q(M,N)$ can be given the structure of a $C^\infty$-smooth Hilbert manifold; this was done first in \cite{Eells1966, Palais1968} and a different but compatible differentiable structure is described in \cite{Inci2013}.

For $q > \frac d2+1$ the diffeomorphism group $\mc D^q(M)$ can be defined by
\begin{align*}
\mc D^q(M) &= \{ \ph \in H^q(M, M) \,:\, \ph \text{ bijective, }
\ph\inv \in H^q(M, M) \} \\
&= \{ \ph \in H^q(M, M) \,:\, 
\ph \in \on{Diff}^1_+(M) \}\,, 
\end{align*}
and $\on{Diff}^1(M)$ denotes the group of $C^1$-diffeomorphisms of $M$. The diffeomorphism group is an open subset of $H^q(M,M)$ and it is a topological group.

We will need the following result on the boundedness of pointwise multiplications and the regularity of the composition map.

\begin{lemma}[{\cite[Lem.~2.13 and Sect.~3]{Inci2013}}]
Let $q, r \in \R$ with $q > \frac d2$ and $0 \leq r \leq q$. Then pointwise multiplication
\[
H^q(M, \R) \times H^r(M,\R) \to H^r(M,\R)\,,\quad
(f, g) \mapsto f \cdot g\,,
\]
is a bounded bilinear map.
\end{lemma}

\begin{lemma}[{\cite[Thm.~1.2 and Rem.~1.5]{Inci2013}}]
\label{thm:diff_M_comp}
Let $q > \frac d2 + 1$ and $k \in \mathbb N$. Then
\[
H^{q+k}(M, N) \times \mc D^q(M) \to H^q(M,N)\,,\quad
(f, \ph) \mapsto f \circ \ph\,,
\]
is a $C^k$-map.
\end{lemma}

Now we can state the analogue of Thm.~\ref{thm:main_euclidean} for Sobolev spaces on manifolds and $\mc D^q(M)$-equivariant maps.

\begin{theorem}
\label{thm:thm_main}
Let $q > \frac d2 + 1$, $0 \leq l \leq k$ and $F : H^q(M, N) \to H^q(M, P)$ be a $\mc D^q(M)$-equivariant $C^k$-map, i.e. $F(u \circ \ph) = F(u) \circ \ph$. Then $F$ maps $H^{q+l}$ into $H^{q+l}$ and $F: H^{q+l}(M, N) \to H^{q+l}(M, P)$ is a $C^{k-l}$-map.
\end{theorem}

\begin{proof}
{\bfseries Step 1.} Reduction to $N=\R^n$ and $P=\R^m$. \\
Using Whitney's embedding theorem we can embed $N$ and $P$ into Euclidean space. Let $N_0 \subseteq \R^{n_0}$ be a tubular neighborhood of $N$ in $\R^{n_0}$ and denote by $\io_N : N \to N_0$ and $r_N : N_0 \to N$ the inclusion and retraction maps. Similarly we introduce the tubular neighborhood $P_0 \subseteq \R^{m_0}$ of $P$ and the maps $\io_P$ and $r_P$. Then we extend $F$ to the map $F_0: H^q(M,N_0) \to H^q(M,P_0)$ via the following commutative diagram
\[
\xymatrix{
H^q(M,N) \ar[r]^F & H^q(M,P) \ar[d]^{u \mapsto \io_P \circ u} \\
H^q(M,N_0) \ar[u]^{u \mapsto r_N \circ u} \ar[r]_{F_0} & H^q(M,P_0)
}
\]
The extension $F_0$ is again $\mc D^q(M)$-equivariant, since
\[
F_0(u \circ \ph) = \io_P \circ F(r_N \circ u \circ \ph) = \io_P \circ F(r_N \circ u) \circ \ph
= F_0(u) \circ \ph\,.
\]
We note that $H^q(M,N_0)$ and $H^q(M,P_0)$ are open subsets of $H^q(M,\R^{n_0})$ and $H^q(M,\R^{m_0})$ respectively. If the theorem is proven in the case, when $N$ and $P$ are the Euclidean space, then together with Rem.~\ref{rem:local_M} this shows that $F_0 : H^{q+l}(M,N_0) \to H^{q+l}(M,P_0)$ is $C^{k-l}$. Now we write
\[
F(u) = r_P \circ \io_P \circ F(r_N \circ \io_N \circ u) = r_p \circ F_0(\io_N \circ u)\,.
\]
Since composition from the left with $C^\infty$-functions $r_P, \io_N$ are $C^\infty$-maps on Sobolev spaces, it follows that $F : H^{q+l}(M,N) \to H^{q+l}(M,P)$ is $C^{k-l}$.

For the rest of the proof we will assume that $N=\R^n$ and $P = \R^m$.

{\bfseries Step 2.} If $F:H^q \to H^q$ is $C^1$, then $F: H^{q+1} \to H^{q+1}$ is $C^0$. \\
Choose smooth vector fields $X_1,\dots,X_A \in \mf X(M)$ such that 
\[
\on{span}\{X_1(x),\dots,X_A(x)\} = T_xM\,,
\] 
for all $x \in M$. Then an equivalent norm for $H^q(M,\R^n)$ is given by
\[
\|u \|_{H^q} \sim \| u \|_{H^{q-1}} + \sum_{j=1}^A \| Tu.X_j \|_{H^{q-1}}\,.
\]

Let $\ps_t = \on{exp}(t.X_j)$, where $\on{exp}$ denotes the Lie group exponential on $\mc D^q(M)$. Then $t \mapsto \ps_t$ is a one-parameter subgroup and the map
\[
\R \x H^{q+1} \to H^q\,,\quad(t, u) \mapsto u \circ \ps_t
\]
is $C^1$. Now fix $u \in H^{q+1}$ and consider the identity $F(u \circ \ps_t) = F(u) \circ \ps_t$. Both sides of the identity are $C^1$ as maps $\R \to H^{q-1}$ and by differentiating at $t=0$ we obtain the identity
\[
DF(u).(Tu.X_j) = T\left(F(u)\right). X_j\,,
\]
We can estimate the $H^q$-norm of the right hand side via the left hand side,
\[
\| T\left(F(u)\right).X_j\|_{H^q} \leq \|DF(u)\|_{\mc L(H^q, H^q)} \| Tu.X_j\|_{H^q}\,,
\]
and hence we see using the $H^q$-norm from above that
\[
\| F(u) \|_{H^{q+1}} \lesssim \| F(u) \|_{H^q} + \| DF(u) \|_{\mc L(H^q, H^q)}\,.
\| u \|_{H^{q+1}}
\]
This shows that $F(u) \in H^{q+1}$, provided $u \in H^{q+1}$. Regarding continuity we can show in the same manner the estimate
\begin{multline*}
\| F(u) -F(v) \|_{H^{q+1}} \lesssim  \| F(u) - F(v) \|_{H^q} + \\ +
\| DF(u) -DF(v) \|_{\mc L(H^q, H^q)} \| u \|_{H^{q+1}} +  \| DF(v) \|_{\mc L(H^q, H^q)} \| u -v\|_{H^{q+1}}
\end{multline*}
for $u, v \in H^{q+1}$, from which the continuity of $F: H^{q+1} \to H^{q+1}$ follows.

The rest of the proof follows in the same way as steps 2 and 3 of the proof of Thm.~\ref{thm:main_euclidean}.
\end{proof}

\begin{remark}
\label{rem:local_M}
As in the previous section, all the arguments in the proof of the theorem are local in nature, i.e. the one-parameter subgroups $\on{exp}(t.\om_i.h_j)$ are only considered around $t=0$ and the statement about differentiability itself is local. Hence the theorem continues to hold for functions defined on open subsets of Sobolev spaces and because the local version is used implicitely in the proof we state it below.
\end{remark}

\begin{corollary}
\label{cor:local_M}
Let $q > \frac d2 + 1$, $0 \leq l \leq k$, $V \subseteq H^q(M,N)$ an open subset and $F : V \to H^q(M, P)$ a $\mc D^q(M)$-equivariant $C^k$-map, i.e. $F(u \circ \ph) = F(u) \circ \ph$, whenever $u, u \circ \ph \in V$. Then $F$ maps $U = V \cap H^{q+l}$ into $H^{q+l}$ and $F:U \to H^{q+l}(M, P)$ is a $C^{k-l}$-map.
\end{corollary}

\section{Geodesic Boundary Value Problem on the Diffeomorphism Group}
\label{sec:diff}

In this section we assume that $M$ is $\R^d$ or a compact manifold without boundary of dimension $d$ and $q \in \R$ with $q > \frac d2 + 1$.

The group $\mc D^q(M)$ introduced in the previous sections is a smooth Hilbert manifold and a topological group. Let $G$ be a smooth right-invariant metric on $\mc D^q(M)$, i.e.
\[
G_\ph(X, Y) = \langle X \circ \ph\inv, Y \circ \ph\inv \rangle\,,
\]
for some fixed inner product $\langle \cdot, \cdot \rangle$ on $\mf X^q(M)$ and $X,Y \in T_\ph \mc D^q(M)$. Note that $\langle \cdot,\cdot \rangle$ does not necessarily have to induce the Sobolev topology on $\mf X^q(M)$. If the geodesic spray of the metric is a smooth vector field on $T\mc D^q(M)$, then it admits an exponential map
\[
\on{Exp} : T\mc D^q(M) \supseteq \mc U \to \mc D^q(M)\,,
\]
defined on an open neighborhood $\mc U$ of the zero section. Because the metric is right-invariant, the geodesic spray and the exponential map are $\mc D^q(M)$-equivariant, i.e.
\[
\on{Exp}(X \circ \ps) = \on{Exp}(X) \circ \ps\,,
\]
holds for $X, X \circ \ps \in \mc U$. 

Because $\mc D^q(M)$ is not a Lie group -- in particular the inverse map $\ph \mapsto \ph\inv$ is continuous but not differentiable -- not every inner product $\langle\cdot,\cdot\rangle$ leads to a smooth right-invariant metric $G$. Similarly, because the topology induced by the metric can be weaker than the manifold topology, not every smooth metric admits an exponential map. Hence the assumption, that $G$ is a smooth metric with a smooth exponential map, are nontrivial.

\subsection{Metrics with smooth sprays}
At this point we should give examples of metrics to which this discussion applies.

Let $(M,g)$ be a compact Riemannian manifold and $d\mu$ the induced volume form. It is shown in \cite{GayBalmaz2009} that the geodesic spray of the $H^1$-metric
\[
\langle u, v \rangle = \int_M g(u,v) + g(\nabla u, \nabla v)\, d\mu
\]
is smooth on $\mc D^q(M)$. In \cite{Modin2015} the smoothness of the spray is shown for a more general family of $H^1$-metrics, defined using the Hodge decomposition of vector fields.

If $n > \frac d2 + 1$ is an integer and
\[
A=(\on{Id} + \De^n) \text{ or }
A=(\on{Id} + \De)^n\,,
\]
where $\De u = (\de du^\flat + d\de u^\flat)^\sharp$ is the positive definite Hodge Laplacian or some other combination of intrinsically defined differential operators with smooth coefficient functions, such that $A$ is positive and elliptic, then the metric induced by the inner product
\[
\langle u, v \rangle = \int_M g(Au,v) d\mu\,,
\]
is smooth on $\mc D^n(M)$; see \cite{Ebin1970} and \cite{Misiolek2010} for details.

Let $M=\R^d$ and $A$ be a Fourier multiplier of order $2s \geq 1$ with $s \in \R$, i.e.
\[
\mc F\left(Au\right)(\xi) = a(\xi) \mc F u(\xi)\,,
\]
for some function $a(\xi)$ satisfying certain asymptotic and ellipticity conditions, then the metric induced by the inner product
\[
\langle u, v \rangle = \int_{\R^d} Au \cdot v \ud x\,,
\]
is smooth on $\mc D^q(\R^d)$ provided $q \geq s$ and $q > \frac d2 + 1$. The same is true for metrics defined by Fourier multipliers on $\mc D^q(S^1)$; see \cite{Bauer2015} and \cite{Escher2014}.

\subsection{Initial value problem}
The smoothness and equivariance of the exponential map together imply via Thm.~\ref{thm:thm_main} that for $k \geq 0$ the map
\[
\on{Exp}: \mc U \cap T \mc D^{q+k}(M) \to \mc D^{q+k}(M)
\]
is also smooth. To see that we can apply the theorem note that $TH^q(M,M) = H^q(M,TM)$ and hence $T\mc D^q(M) \subset H^q(M,TM)$ is an open subset. This also holds in the smooth category ($k=\infty$) and the map
\[
\on{Exp} : \mc U \cap T\on{Diff}_{H^\infty}(M) \to \on{Diff}_{H^\infty}(M)
\]
is smooth. When $M$ is compact $\on{Diff}_{H^\infty}(M) = \on{Diff}(M)$ is the group of smooth diffeomorphisms; for $M=\R^d$ it is the intersection $\on{Diff}_{H^\infty}(\R^d) = \bigcap_{k > \frac q2 + 1} \mc D^q(\R^d)$.

This can be interpreted as a ``no loss of regularity''-result for the geodesic equation on $\mc D^q(M)$. If the initial conditions $\ph_0, X_0$ of a geodesic are of class $H^{q+k}$ with $k \geq 0$, then the whole geodesic $\ph(t) = \on{Exp}(t.X_0)$ is also of class $H^{q+k}$. In other words, the geodesic is at least as smooth as its initial conditions. Because the geodesic equation can be solved forward and backward in time we also have a version of a ``no loss of regularity''-result: if for some time $t > 0$, both $\ph(t)$ and $\p_t \ph(t)$ are of class $H^{q+k}$, then the initial conditions $\ph_0, X_0$ must already be of class $H^{q+k}$

The property of the exponential map to preserve smoothness was first observed in \cite{Ebin1970}; for a detailed exposition see \cite{Escher2014}. However, for the proofs in the above references one needs to know a priori that for a given metric the geodesic spray is smooth on all groups $\mc D^{q+k}(M)$ with $k \in \mathbb N$. One advantage of Thm.~\ref{thm:thm_main} is that the smoothness of the spray on $\mc D^q(M)$ alone already implies its smoothness on $\mc D^{q+k}(M)$ for all $k \geq 0$.

\subsection{Boundary value problem}
Of interest is the following inverse problem: given a geodesic, is the initial velocity at least as smooth as the boundary diffeomorphisms? 

If $\ph(t)$ is a geodesic in $\mc D^q(M)$ and $\ph(0), \ph(1) \in \mc D^{q+k}(M)$ for some $k \geq 0$, does it follow that $\ph(t) \in \mc D^{q+k}(M)$ for $0 < t < 1$? Note that this is different from the ``no loss of regularity''-result above, which required both $\ph(1)$ and $\p_t \ph(1)$ to be of class $H^{q+k}$. We provide an affirmative answer under the assumption that $\ph(0)$ and $\ph(1)$ are nonconjugate along $\ph(t)$.

A first answer was given in \cite{Kappeler2008} for $H^k$-metrics ($k\geq 1$ on $\mc D^q(S^1 \x S^1)$ with $q$ sufficiently large, provided the geodesic remains in a sufficiently small neighborhood around the identity. The following application of Thm.~\ref{thm:thm_main} allows us to extend this result to a wider class of metrics on compact manifolds with fewer restrictions on the order and the geodesic.

\begin{proposition}
\label{prop:reg_diff}
Let $G$ be a smooth right-invariant metric on $\mc D^q(M)$ with a smooth exponential map
$\on{Exp} : T\mc D^q(M) \supseteq \mc U \to \mc D^q(M)$ and $\ph_0, \ph_1 \in \mc D^q(M)$, $X \in T_{\ph_0}\mc D^q(M)$. If $\ph_1 = \on{Exp}(X)$ and $\ph_0, \ph_1 \in \mc D^{q+k}(M)$ for some $k \geq 0$ and $\pi_{\mc D^q(M)} \x \on{Exp}: T \mc D^q(M) \to \mc D^q(M) \x \mc D^q(M)$ is locally invertible around $X$, then $X \in T\mc D^{q+k}(M)$.
\end{proposition}

\begin{proof}
We have $\pi_{\mc D^q(M)} \x \on{Exp}(X) = (\ph_0, \ph_1)$. Let
\[
\on{Log} : \mc V_0 \x \mc V_1 \to T \mc D^q(M)
\]
be the local inverse of $\pi_{\mc D^q(M)} \x \on{Exp}$ with $\mc V_i \subseteq \mc D^q(M)$ an open neighborhood of $\ph_i$. Because $\on{Exp}$ and $\pi_{\mc D^q(M)}$ are $\mc D^q(M)$-equivariant, so is $\on{Log}$,
\[
\on{Log}(\et_0 \circ \ps, \et_1 \circ \ps) = \on{Log}(\et_0, \et_1) \circ \ps\,,
\]
for $\et_i \in \mc V_i$ and $\ps \in \mc D^q(M)$ such that $\et_i \circ \ps \in \mc V_i$. Now we apply Thm.~\ref{thm:thm_main} -- note that $\on{Log}$ is a $C^\infty$-function -- to conclude that
\[
\on{Log}: \left(\mc V_0 \cap \mc D^{q+k}(M)\right) \x \left(\mc V_1 \cap \mc D^{q+k}(M)\right)
\to T\mc D^{q+k}(M)
\]
is also a $C^\infty$-function. Since $\ph_0, \ph_1 \in \mc D^{q+k}(M)$, we obtain that $X = \on{Log}(\ph_0, \ph_1) \in T\mc D^{q+k}(M)$.
\end{proof}

For a $C^\infty$-smooth exponential map we can also make a statement about the group $\on{Diff}(M)$ of $C^\infty$-smooth diffeomorphisms. If $M=\R^d$, then we have to consider the group $\on{Diff}_{H^\infty}(\R^d) = \bigcap_{q > \frac d2 + 1} \mc D^q(\R^d)$ of diffeomorphisms that decay to $\on{Id}$ like an $H^\infty$-function.

\begin{corollary}
Let $G$ be a smooth right-invariant metric on $\mc D^q(M)$ with a smooth exponential map. If $\ph_0, \ph_1 \in \on{Diff}(M)$ (or $\on{Diff}_{H^\infty}(\R^d)$ for $M=\R^d$) are nonconjugate along the geodesic $\ph(t)$, then $\ph(t) \in \on{Diff}(M)$ for all $t \in [0,1]$.
\end{corollary}

\begin{proof}
The geodesic is given by $\ph(t) = \on{Exp}(t \p_t \ph(0))$ and Prop.~\ref{prop:reg_diff} shows that $\p_t \ph(0) \in T_{\ph_0}\on{Diff}(M)$. Hence $\ph(t) \in \on{Diff}(M)$ as well.
\end{proof}

\section{Geodesic Boundary Value Problem on $\on{Imm}(S^1,\R^d)$}
\label{sec:curves}

We can apply the same idea to show a result about the regularity of the geodesic boundary value problem on the space of curves. Let $d \geq 2$. We consider the space
\[
\on{Imm}(S^1,\R^d) = \left\{ c \in C^\infty(S^1,\R^d) \,:\, c'(\th) \neq 0\,,\; \forall \th \in S^1 \right\}
\]
of immersions or smooth regular curves and for $q > 3/2$ also the space of Sobolev curves
\[
\mc I^q(S^1,\R^d) = \left\{ c \in H^1(S^1,\R^d) \,:\, c'(\th) \neq 0\,,\; \forall \th \in S^1 \right\}\,,
\]
together with the family of reparametrization invariant Sobolev metrics
\[
G_c(h,k) = \int_{S^1} a_0 \langle h, k \rangle + \dots + 
a_n \langle D_s^n h, D_s^k \rangle \ud s\,;
\]
here $D_s h = \frac{1}{|c'|} h'$ and $\ud s = |c'|\ud \th$ denote differentiation and integration with respect to arc-length respectively. The coefficients $a_j \geq 0$ are assumed to be constant with $a_0, a_n > 0$, $n$ is the order of the metric and $h, k \in T_c\on{Imm}(S^1,\R^d)$ are tangent vectors at $c$.

The space $\on{Imm}(S^1,\R^d)$ is an open set in the Fr\'echet space $C^\infty(S^1, \R^d)$ with respect to the $C^\infty$-topology and $\mc I^q(S^1,\R^d)$ is open in $H^q(S^1,\R^d)$. As open subsets of vector spaces the tangent bundles of the spaces $\on{Imm}(S^1,\R^d)$ and $\mc I^q(S^1,\R^d)$ are trivial,
\begin{align*}
T\on{Imm}(S^1,\R^d) &\cong \on{Imm}(S^1,\R^d) \x C^\infty(S^1,\R^d) \\
T\mc I^q(S^1,\R^d) &\cong \mc I^q(S^1,\R^d) \x H^q(S^1,\R^d)\,.
\end{align*}

Sobolev metrics  on curves were first introduced in \cite{Charpiat2007,Mennucci2007,Michor2006c}. They were generalized to immersed higher-dimensional manifolds in \cite{Bauer2011b}. See \cite{Bauer2014} for an overview of their properties and how they relate to other metrics used in shape analysis.

For $n\geq 2$ the metric $G$ of order $n$ can be extended to a smooth Riemannian metric on $\mc I^n(S^1,\R^d)$; then $(\mc I^n(S^1,\R^d), G)$ is a strong Riemannian manifold and $G$ is $\mc D^n(S^1)$-invariant. This was first observed in \cite[3.2]{Bruveris2014}.

A $\mc D^n(S^1)$-invariant metric has a $\mc D^n(S^1)$-invariant spray and therefore the exponential map
\[
\on{Exp} : \mc I^n(S^1,\R^d) \times H^n(S^1,\R^d) \to \mc I^n(S^1,\R^d)
\]
is $\mc D^n(S^1)$-equivariant. It is shown in \cite[Thm.~5.5]{Bruveris2014} and \cite[Thm.~4.3]{Bruveris2015} that the manifold $(\mc I^n(S^1,\R^d), G)$ is geodesically complete and hence $\on{Exp}$ is defined globally. Thus Thm.~\ref{thm:thm_main} implies the following.

\begin{proposition}
Let $n \geq 2$ and $G$ be a Sobolev metric of order $n$ with constant coefficients. Then for all $k \geq 0$,
\[
\on{Exp} : \mc I^{n+k}(S^1,\R^d) \times H^{n+k}(S^1,\R^d) \to \mc I^{n+k}(S^1,\R^d)
\]
is a $C^\infty$-map. In particular, if $c_0 \in \mc I^{n+k}(S^1,\R^d)$ and $u \in H^{n+k}(S^1,\R^d)$, then $\on{Exp}(c_0, u) \in \mc I^{n+k}(S^1,\R^d)$.

The proposition also holds for smooth curves, i.e. $k=\infty$.
\end{proposition}

This proposition has been proven directly in \cite[Thm.~3.7.]{Bruveris2014} and under slightly stronger assumptions in \cite[Thm.~4.3.]{Michor2007}. Informally it states that the geodesic with respect to a Sobolev metric is at least as smooth as the initial curve and initial velocity. We are interested in the reverse implication: is the initial velocity of a geodesic at least as smooth as both endpoints? While we cannot give an unconditionally affirmative answer, we can do so under the assumption that the endpoints are non-conjugate along the given geodesic.

\begin{proposition}
Let $n \geq 2$ and $G$ be a Sobolev metric of order $n$ with constant coefficients. If $c_1 = \on{Exp}(c_0, u)$ and $c_0, c_1 \in \mc I^{n+k}(S^1,\R^d)$ for some $k \geq 0$ and $D_u \on{Exp}(c_0, u) : T_{c_0} \mc I^n \to T_{c_1} \mc I^n$ is invertible, then $u \in H^{n+k}(S^1,\R^d)$.
\end{proposition}

\begin{proof}
The Riemannian metric $G$ is $\mc D^n(S^1)$-invariant and so its exponential map
\[
\on{Exp} : T\mc I^n(S^1,\R^d) \to \mc I^n(S^1,\R^d)
\]
is $\mc D^n(S^1)$-equivariant,
\[
\on{Exp}(c \circ \ph, v \circ \ph) = \on{Exp}(c, v) \circ \ph\,,
\]
for $v \in T_c\mc I^n(S^1,\R^d)$ and $\ph \in \mc D^n(S^1)$. Since $D_u \on{Exp}(c_0, u)$ is assumed to be invertible, it follows that the map
$\on{Id}_{\mc I^n} \times \on{Exp} : \mc I^n \times H^n \to \mc I^n \x \mc I^n$ is locally invertible around $(c_0, u)$.
Let
\[
\on{Log} : \mc U_0 \x \mc U_1 \to \mc I^n(S^1,\R^d) \times H^n(S^1,\R^d)
\]
be the local inverse, with $\mc U_i \subseteq \mc I^n(S^1,\R^d)$ an open neighborhood of $c_i$. Because $\on{Exp}$ and $\on{Id}_{\mc I^n}$ are $\mc D^n(S^1)$-equivariant, so is $\on{Log}$,
\[
\on{Log}(c_0 \circ \ph, c_1 \circ \ph) = \on{Log}(c_0, c_1) \circ \ph\,,
\]
for $c_i \in \mc U_i$ and $\ph \in \mc D^n(S^1)$ such that $c_i \circ \ph \in \mc U_i$. Now we apply Thm.~\ref{thm:thm_main} -- note that $\on{Log}$ is a $C^\infty$-function -- to conclude that
\[
\on{Log}: \left(\mc U_0 \cap \mc I^{n+k}(S^1,\R^d) \right) \x \left(\mc U_1 \cap \mc I^{n+k}(S^1,\R^d) \right)
\to \mc I^{n+k}(S^1,\R^d) \times H^{q+k}(S^1,\R^d)
\]
is also a $C^\infty$-function. Since we assumed that $c_0, c_1 \in \mc I^{n+k}(S^1,\R^d)$, we obtain that $u = \pi_2 \circ \on{Log}(c_0, c_1) \in H^{n+k}(S^1,\R^d)$.
\end{proof}

If $c_0 \in \mc I^{n+k}(S^1,\R^d)$ and $u \in H^{n+k}(S^1,\R^d)$, then the $\mc D^n(S^1)$-invariance of the exponential map implies that the whole geodesic $c(t) = \on{Exp}(c_0, t.u)$ is $H^{n+k}$-regular as well. Furthermore the above proposition remains valid for $k=\infty$. 

\begin{corollary}
Let $n\geq 2$ and let $G$ be a Sobolev metric of order $n$ with constant coefficients. If $c_0, c_1 \in \on{Imm}(S^1,\R^d)$ are non-conjugate along the geodesic $c(t)$ in $\mc I^n(S^1,\R^d)$, then $c(t) \in \on{Imm}(S^1,\R^d)$ for all $t$.
\end{corollary}

It was shown in \cite[Thm.~5.2.]{Bruveris2015} that any two curves in the same connected component of $\mc I^n(S^1,\R^d)$ can be joined by a minimizing geodesic. The above corollary shows that in the space $\on{Imm}(S^1,\R^d)$ of smooth curves minimizing geodesics exist at least on an open neighborhood of the diagonal in $\on{Imm}(S^1,\R^d) \times \on{Imm}(S^1,\R^d)$.

\let\i=\ii           
\printbibliography

\end{document}